\documentclass[a4paper,12pt]{article}
\usepackage[utf8]{inputenc}
\usepackage{amsfonts}
\usepackage{amssymb}
\usepackage{amsthm}
\usepackage{amsmath}
\usepackage{comment}
\usepackage{xcolor}
\usepackage[margin=3cm]{geometry}
\usepackage{bm}
\usepackage[english]{babel}
\usepackage[numbers,sort, compress]{natbib}
\usepackage[colorlinks=true, allcolors=blue]{hyperref}

\newtheorem{theorem}{Theorem}
\newtheorem{corollary}{Corollary}
\newtheorem{lemma}{Lemma}

\theoremstyle{definition}
\newtheorem{definition}{Definition}
\newtheorem{notation}{Notation}

\newtheorem{remark}{Remark}

\usepackage{titling}
\thanksmarkseries{arabic}

\DeclareMathOperator{\dd}{d}
\DeclareMathOperator{\sep}{sep}
\DeclareMathOperator{\grlex}{grlex}
\DeclareMathOperator{\Frac}{Frac}
\newcommand{\Mer}{\mathcal{M}er}

\title{From algebra to analysis: new proofs of\\ theorems by Ritt and Seidenberg}

\author{D.\,Pavlov\thanks{Faculty of Mechanics and Mathematics, Moscow State University, Russia, e-mail: \url{dmmpav@gmail.com}} \and G.\,Pogudin\thanks{LIX, CNRS, \'Ecole Polytechnique, Institute Polytechnique de Paris, Palaiseau, France, email: \url{gleb.pogudin@polytechnique.edu}} \and Yu.P.\,Razmyslov\thanks{Faculty of Mechanics and Mathematics, Moscow State University, Russia, e-mail: \url{ynona_olga@rambler.ru}}}

\date{}

\begin{document}

\maketitle

\begin{abstract}
    Ritt's theorem of zeroes and Seidenberg's embedding theorem are classical results in differential algebra allowing to connect algebraic and model-theoretic results on nonlinear PDEs to the realm of analysis.
    However, the existing proofs of these results use sophisticated tools from constructive algebra (characteristic set theory) and analysis (Riquier's existence theorem).
    
    In this paper, we give new short proofs for both theorems relying only on basic facts from differential algebra and the classical Cauchy-Kovalevskaya theorem for PDEs.
\end{abstract}

\section{Introduction}

The algebraic theory of differential equations, also known as differential algebra~\cite{Ritt}, aims at studying nonlinear differential equations using methods of algebra and algebraic geometry.
For doing this, one typically abstracts from functions (analytic, meromorphic, etc) to elements of differential fields (fields equipped with a derivation or several commuting derivations).
This approach turned out to be fruitful yielding interesting results from theoretical and applied perspectives (see, e.g., \cite{Fliess1987, Boulier, vanderPut2003,Pila2016, jfunc}).
Furthermore, one can additionally use powerful tools from model theory to study differential fields (see, e.g., \cite{Marker, Nagloo2021}).

In this context, a fundamental question is how to transfer results about differential fields back to the realm of analysis.
There are two classical theorems in differential algebra typically used for this purpose:
\begin{itemize}
    \item \emph{Ritt's theorem of zeroes}~\cite[p. 176]{Ritt} which can be viewed as an analogue of Hilbert's Nullstellensatz. The theorem implies that any system of nonlinear PDEs having a solution in some differential field has a solution in a field of meromorphic functions on some domain.
    \item \emph{Seidenberg's embedding theorem}~\cite{Seid1958} which is often used as a differential analogue of the Lefschetz principle (e.g.~\cite{jfunc, Gauchman1989,Buium1995, Binyamini2017, Hardouin2008}).
    The theorem says that any countably generated differential field with several commuting derivations can be embedded into a field of meromorphic functions on some domain.
\end{itemize}

In~\cite{Seid1958}, Seidenberg gave a complete proof of his theorem for the case of a single derivation (see also~\cite[Appendix~A]{Marker}).
For the PDE case, he gave a sketch which reuses substantial parts of Ritt's proof of Ritt's zero theorem from~\cite{Ritt}.
The latter proof concludes the whole monograph and heavily relies on the techniques developed there.
In particular, Ritt's proof uses the machinery of characteristic sets~\cite[Chapter V]{Ritt} which is a fundamental tool in differential algebra but not so well-known in the broader algebra community and quite technical existence theorem for PDEs due to Riquier~\cite[Chapter VIII]{Ritt} (see also~\cite{Riquier}) which, to the best of our knowledge, is not discussed in the standard PDE textbooks.

Due to the importance of the theorems of Ritt and Siedenberg as bridges between the algebraic and analytic theories of nonlinear PDEs, we think that it is highly desirable to have short proofs of these theorems accessible to people with some general knowledge in algebra and PDEs.
In the present paper, we give such proofs.
Our proofs rely only on some basic facts from differential algebra and the classical Cauchy-Kovalevskaya theorem for PDEs.

Our proof strategy is inspired by the argument from~\cite[Theorem~3.1]{GRP} for the case of one derivation.
However, the techniques from~\cite{GRP} had to be substantially developed in order to tackle the PDE case (which is quite subtle~\cite{Lemaire2003}) and to prove both Ritt's and Seidenberg's theorem (not only the Ritt's as in~\cite{GRP}).
The key ingredients of the argument are an auxiliary change of derivations (Lemma~\ref{lemmacoef}) which helps us to bring a system of PDEs into the form as in the Cauchy-Kovalevaskaya theorem, Taylor homomorpishms (Definition~\ref{deftaylor}) allowing to build formal power series solutions, and a characterization of differentially simple algebras (Lemma~\ref{lemmasimple}).

The paper is organized as follows. 
Section~\ref{sec:preliminaries} contains the basic definitions used to state the main results in Section~\ref{sec:main}.
Section~\ref{sec:proofs_notions} contains relevant notions and facts from algebra and analysis used in the proofs.
The proofs are located in Section~\ref{sec:proofs}.
Section~\ref{sec:spec} contains a remark on the special case of algebras over $\mathbb{C}$.


\section{Preliminaries}\label{sec:preliminaries}

\subsection{Algebra}

Throughout the paper, all algebras are assumed to be \emph{unital} (that is, with a multiplicative identity element).

\begin{notation}[Multi-indices]
  For every $\alpha = (\alpha_1, \ldots, \alpha_m) \in \mathbb{Z}_{\geqslant 0}^m$ and for every tuple $t = (t_1, \ldots, t_m)$ of elements of a ring, we denote
  \[
    t^{\alpha} := t_1^{\alpha_1}\cdot \ldots \cdot t^{\alpha_m} \quad\text{ and }\quad \alpha! := \alpha_1!\cdot \ldots \cdot \alpha_m!.
  \]
\end{notation}

\begin{definition}[Differential rings and algebras]
Let $\Delta = \{ \delta_1, \ldots, \delta_m\}$ be a set of symbols.
\begin{itemize}
    \item Let $R$ be a commutative ring. An additive map $\delta \colon R \to R$ is called \emph{derivation} if $\delta(ab) = \delta(a) b + a \delta(b)$ for any~$a,b\in R$.
    \item A commutative ring $R$ is called \emph{$\Delta$-ring} if $\delta_1, \ldots, \delta_m$ act on $R$ as pairwise commuting derivations.
    If $R$ is a field, it is called \emph{$\Delta$-field}.
    
    \item Let $A$ be a commutative algebra over ring $R$.
    If $A$ and $R$ are $\Delta$-rings and the action of $\Delta$ on $R$ coincides with the restriction of the action of $\Delta$ on $R\cdot 1_A \subseteq A$, then $A$  is called \emph{$\Delta$-algebra} over $R$.
\end{itemize}
\end{definition}

\begin{definition}[Differential generators]
  Let $A$ be a $\Delta$-algebra over a $\Delta$-ring $R$.
  A set $S \subseteq A$ is called a set of \emph{$\Delta$-generators} of $A$ over $R$ if the set
  \[
      \{ \delta^{\alpha} s \mid s \in S,\; \alpha \in \mathbb{Z}_{\geqslant 0}^m\}
  \]
  of all the derivatives of all the elements of $S$ generates $A$ as $R$-algebra.
  A $\Delta$-algebra is said to be $\Delta$-finitely generated if it has a finite set of $\Delta$-generators. 
  
  $\Delta$-generators for $\Delta$-fields are defined analogously.
\end{definition}

\begin{definition}[Differential homomorphisms]
Let~$A$ and~$B$ be $\Delta$-algebras over $\Delta$-ring~$R$. 
A map~$f\colon A \rightarrow B$ is called~\emph{$\Delta$-homomorphism} if~$f$ is a homomorphism of commutative~$R$-algebras and~$f(\delta a) = \delta f(a)$ for all~$\delta \in \Delta$ and~$a\in A$. 
An injective~$\Delta$-homomorphism is called \emph{$\Delta$-embedding}.
\end{definition}

\begin{definition}[Differential algebraicity]
Let $A$ be a $\Delta$-algebra over a $\Delta$-ring $R$.
An element~$a\in A$ is said to be~\emph{$\Delta$-algebraic} over~$R$ if the set~$\{ \delta^{\alpha}a \mid \alpha \in \mathbb{Z}_{\geqslant 0}^m \}$ of all the derivatives of $a$ is algebraically dependent over~$R$.

In other words, $a$ satisfies a nonlinear PDE with coefficients in $R$.
\end{definition}


\subsection{Analysis}
\begin{definition} [Multivariate holomorphic functions]
Let $U \subseteq \mathbb{C}^m$ be a domain.
A function~$f: U \rightarrow \mathbb{C}$ is called a \emph{holomorphic} function in~$m$ variables on $U$ if it is holomorphic on~$U$ with respect to each individual variable.
The set of all holomorphic functions on $U$ will be denoted by $\mathcal{O}_m(U)$
\end{definition}

\begin{notation}
  Let~$f$ be a holomorphic function on~$U \subseteq \mathbb{C}^m$. 
  By~$V(f)$ we denote the set of zeroes of~$f$.
\end{notation}

\begin{definition}[{Multivariate meromorphic functions, \cite[Chapter IV, Definition 2.1]{FL}}]
Let $U\subseteq \mathbb{C}^m$ be a domain.
A \emph{meromorphic} function on $U$ is a pair $(f, M)$, where~$M$ is a thin set in~$U$ and~$f \in \mathcal{O}_m(U\setminus M)$ with the following property: for every~$z_0\in U$, there is a neighbourhood~$U_0$ of~$z_0$ and there are functions $g, h \in \mathcal{O}_m(U_0)$, such that~$V(h)\subseteq M$ and 
\[
  f(z) = \dfrac{g(z)}{h(z)}~\text{ for every }~z\in U_0 \setminus M.
\]
The set of meromorphic functions on a domain~$U$ is denoted~$\mathcal{M}er_m(U)$. 
By convention we define $\mathcal{M}er_0(U) = \mathcal{O}_0(U) = \mathbb{C}$.

For every domain $U \subseteq \mathbb{C}^m$, the field $\Mer_m(U)$ has a natural structure of $\Delta$-field with $\delta_i \in \Delta$ acting as $\frac{\partial}{\partial z_i}$, where $z_1, \ldots, z_m$ are the coordinates in $\mathbb{C}^m$.
Furthermore, if $U \subseteq V$, then there is a natural $\Delta$-embedding $\Mer_m(V) \subseteq \Mer_m(U)$.
\end{definition}


\section{Main Results}\label{sec:main}

\begin{theorem}[Seidenberg's embedding theorem]
Let $W \subseteq \mathbb{C}^m$ be a domain and
let~$K \subseteq \mathcal{M}er_m(W)$ be at most countably $\Delta$-generated~$\Delta$-field (over $\mathbb{Q}$).
Let $L \supset K$ be a $\Delta$-field finitely $\Delta$-generated over~$K$. 

Then there exists a domain $U \subseteq W$ and a~$\Delta$-embedding~$f\colon L \rightarrow \mathcal{M}er_m(U)$ over $K$. 
\end{theorem}

\begin{theorem}[Ritt's theorem of zeroes]
Let $W \subseteq \mathbb{C}^m$ be a domain and
let~$K \subseteq \mathcal{M}er_m(W)$ be a~$\Delta$-field. 
Let~$A$ be a finitely generated~$\Delta$-algebra over~$K$. 

Then there exists a non-trivial $\Delta$-homomorphism~$f: A \rightarrow \mathcal{M}er_m(U)$ for some domain ${U \subseteq W \subseteq \mathbb{C}^m}$ such that $f(a)$ is~$\Delta$-algebraic over~$K$ for any~$a\in A$.
\end{theorem}

\begin{corollary}\label{cor:holomorphic}
  Let~$A$ be a finitely $\Delta$-generated~$\Delta$-algebra over~$\mathbb{C}$. Then there exists a non-trivial $\Delta$-homomorphism~$f\colon A \rightarrow \mathcal{O}_m(U)$ for some domain ${U \subseteq \mathbb{C}^m}$. 
\end{corollary}

\begin{proof}
  Ritt's theorem yields the existence of a $\Delta$-homomorphism $f\colon A\to \Mer_m(W)$.
  Let $a_1, \ldots, a_n$ be a set of $\Delta$-generators of $A$.
  There is a domain $U \subseteq W$ such that $f(a_1), \ldots, f(a_n)$ are holomorphic in $U$.
  Therefore, the restriction of $f$ to $U$ yields a $\Delta$-homomorphism $A \to \mathcal{O}_m(U)$.
\end{proof}


\section{Notions and results used in the proofs}\label{sec:proofs_notions}

\subsection{Algebra}


\begin{notation}
  Let $R$ be a $\Delta$-ring.
  By~$R[[z_1,\ldots,z_m]]$ we denote the ring of formal power series over~$R$ in variables~$z_1,\ldots,z_m$.
  It has a natural structure of $\Delta$-algebra over $R$ with $\delta_i \in \Delta$ acting as $\frac{\partial}{\partial z_i}$.
\end{notation}

\begin{definition}[Taylor homomorphisms]\label{deftaylor}
Let $A$ be a $\Delta$-algebra over $\Delta$-field $K$, $L \supseteq K$ be a~$\Delta$-field and the action of $\Delta$ on $L$ be trivial.
Let $\psi\colon A \to L$ be a (not necessarily differential) homomorphism of $K$-algebras. Let~$w\in L^m$.
Then we define a map called \emph{Taylor homomorphism} $T_{\psi, w} \colon A \to L[[t_1,\ldots,t_m]]$ by the formula
\[
  T_{\psi, w}(a) := \sum\limits_{\alpha \in \mathbb{Z}^m_{\geqslant 0}} \psi(\delta^\alpha a)\dfrac{(t-w)^\alpha}{\alpha!} \quad \text{ for every } a\in A.
\]
Direct computation shows~\cite[\S 44.3]{YuPbook} that that $T_{\psi, w}$ is a $\Delta$-homomorphism.
\end{definition}

\begin{notation}
  Let $R$ be a $\Delta$-ring.
  For every subset $S \subseteq R$, by~$\Delta^\infty S$ we denote the set $\{\delta^\alpha s | \alpha\in \mathbb{Z}^m_{\geqslant 0}, s\in S\}$ of all derivatives of the elements of~$S$.
\end{notation}

\begin{definition}[Differential polynomials]
Let $R$ be a $\Delta$-ring.
Consider an algebra of polynomials 
\[
 R[\Delta^\infty x_1,\ldots, \Delta^\infty x_n] := R[\delta^\alpha x_i| \alpha\in\mathbb{Z}^m_{\geqslant 0}, i=1,\ldots,n]
\]
in infinitely many variables $\delta^\alpha x_i$.
We define the structure of $\Delta$-algebra over $R$ by
\[
  \delta_i (\delta^\alpha x_j) := (\delta_i \delta^\alpha) x_j \text{ for every } 1\leqslant i \leqslant m,\; 1\leqslant j \leqslant n,\; \alpha\in \mathbb{Z}_{\geqslant 0}^m.
\]
The resulting algebra is called the \emph{the algebra of $\Delta$-polynomials} in $x_1, \ldots, x_n$ over $R$.
\end{definition}

\begin{definition}[Separants] \label{defsepinit}
Let $R$ be a $\Delta$-ring.
Let~$P(x) \in R[\Delta^\infty x]$. 
We introduce an~ordering on the derivatives of~$x$ as~follows:
\begin{equation}\label{eq:ord}
  \delta^\alpha x < \delta^\beta x\iff \alpha <_{\grlex} \beta,
\end{equation}
where~$\grlex$ is the graded lexicographic ordering of~$\mathbb{Z}^m_{\geqslant 0}$.
Let $\delta^\mu x$ be the highest (w.r.t. the introduced ordering) derivative appearing in~$P$. 
Consider~$P$ as a univariate polynomial in~$\delta^{\mu}x$ over~$R[\delta^\alpha x| \alpha <_{\grlex} \mu]$. 
We define the \emph{separant} of $P$ by
\[
  \sep_x^{\Delta}(P) := \frac{\partial}{\partial (\delta^\mu x)} P.
\]
\end{definition}

\begin{remark}
Throughout the rest of the paper, we assume that the ordering of a set of derivatives of an element of a~$\Delta$-algebra is the one defined in~\eqref{eq:ord}.
\end{remark}

\begin{definition}[Differential algebraicity and transcendence]
  Let $R$ be a $\Delta$-ring and let $A$ be a $\Delta$-algebra over $R$.
  \begin{itemize}
      \item A subset~$S\subseteq A$ is said to be \emph{$\Delta$-dependent} over~$R$ if~$\Delta^\infty S$ is algebraically dependent over~$R$. 
Otherwise, $S$ is called \emph{$\Delta$-independent} over~$R$.
      \item An element~$a\in A$ is said to be~\emph{$\Delta$-algebraic} over~$R$ if the set~$\{a\}$ is $\Delta$-dependent over~$R$. 
Otherwise, $a$ is called $\Delta$-transcendental over $R$.
  \end{itemize}
\end{definition}

\begin{definition}[Differential transcendence degree] Let~$A$ be a~$\Delta$-algebra over field~$K$. 
Any maximal~$\Delta$-independent over~$K$ subset of~$A$ is called a \emph{$\Delta$-transcendence basis} of~$A$ over~$K$. 
The cardinality of a $\Delta$-transcendence basis does not depend on the choice of the basis~\cite[II.9, Theorem~4]{Kolchin} and is called the~\emph{$\Delta$-transcendence degree} of~$A$ over~$K$ (denoted by~$\operatorname{difftrdeg}_K^\Delta A$).
\end{definition}

\begin{definition}[Differential ideals]
  Let~$R$ be a $\Delta$-ring. 
  A subset~$I \subseteq R$ is called a \emph{differential ideal} if it is an ideal of~$A$ considered as a commutative algebra and~$\delta a \in I$ for any~$\delta\in \Delta$ and~$a\in I$.
\end{definition}

\begin{notation}
Throughout the rest of the paper, we use the notation $\Delta_0 := \Delta \setminus \{\delta_1\}$.
\end{notation}


\subsection{Analysis}
The following is a special case of the Cauchy-Kovalevskaya theorem~\cite[Chapter V, \textsection 94]{Goursat} which is sufficient for our purposes.
\begin{theorem}[Cauchy-Kovalevskaya]
Consider holomorphic functions in variables $z_1, \ldots, z_m$.
The operator of differentiation with respect to $z_i$ will be denoted by $\delta_i$ for $i = 1, \ldots, m$.
For a positive integer $r$, we introduce a set of multi-indices $M_r := \{\alpha \in \mathbb{Z}_{\geqslant 0}^m \mid |\alpha| \leqslant r, \alpha_1 < r\}$.
Consider a~PDE in an unknown function $u$
\begin{equation} \label{syskov}
    \delta_1^{r} u = F(z_1, \ldots, z_m; \delta^{\alpha}u \mid \alpha \in M_r),
\end{equation}
where $F$ is a rational function over $\mathbb{C}$ in $z_1, \ldots, z_m$ and derivatives $\{\delta^\alpha u \mid \alpha \in M_r\}$.

Consider complex numbers $a_1, \ldots, a_m$ and functions $\varphi_0, \ldots, \varphi_{r - 1}$ in variables $z_2, \ldots, z_m$ holomorphic in a neighbourhood of $(a_2, \ldots, a_m)$ such that $F$ is well-defined under the substitution:
\begin{enumerate}
    \item $a_i$ for $z_i$ for every $1 \leqslant i \leqslant m$
    \item and $(\delta^{(\alpha_2, \ldots, \alpha_m)}\varphi_{\alpha_1})(a_2, \ldots, a_m)$ for $\delta^\alpha u$ for every $\alpha \in M_r$.
\end{enumerate}
Then there is a unique function $u$ holomorphic in a neighborhood of $(a_1, \ldots, a_m)$ satisfying~\eqref{syskov} and
\[
    (\delta_1^i u)|_{z_1 = a_1} = \varphi_i \quad\text{ for every }\quad 0 \leqslant i < r.
\]
\end{theorem}

\section{Proofs}\label{sec:proofs}

This section is structured as follows.
In Section~\ref{sec:dintegral}, we introduce the notion of $\Delta$-integral elements which is an algebraic way saying that an element satisfies a PDE as in the Cauchy-Kovalevskaya theorem. We prove that there always exists a linear change of derivations making a fixed element $\Delta$-integral (Lemma~\ref{lemmacoef}) and prove Lemma~\ref{lemmafinite} which is a key tool for reducing the problem to the same problem in fewer derivations.

Section~\ref{sec:seidenberg} contains the proof of Seidenberg's embedding theorem which proceeds by induction on the number of derivations using Lemma~\ref{lemmafinite}.
We deduce Ritt's theorem of zeroes in Section~\ref{sec:ritt} from Seidenberg's theorem and Lemma~\ref{lemmasimple} characterizing $\Delta$-simple algebras.


\subsection{Differentially integral generators}\label{sec:dintegral}

\begin{definition}[$\Delta$-integral elements]
Let $R$ be a $\Delta$-ring and let $A$ be a $\Delta$-algebra over $R$.
An element~$a\in A$ is said to be~\emph{$\Delta$-integral} over~$R$ if there exists $P(x) \in R[\Delta^\infty x]$ such that 
\begin{itemize}
    \item $P(a)=0$, $\sep_x^\Delta (P) (a) \neq 0$;
    \item the highest (w.r.t. the ordering~\eqref{eq:ord}) derivative in~$P$ is of the form~$\delta_1^r x$.
\end{itemize}
\end{definition}

\begin{remark} \label{eqremark}
If~$a\in A$ is~$\Delta$-integral over~$R$, then the equality~$\delta_1 (P(a)) = 0$ can be rewritten as
\[
  \sep_x^\Delta P (a) \cdot \delta_1^{r+1}a = q(a), \quad \text{where } q\in R[\delta^\alpha x \mid \alpha <_{\grlex} (r + 1, 0, \ldots, 0)].
\]
Therefore, if $\sep_x^\Delta P (a)$ is invertible in $A$, we have~$\delta_1^{r+1}a = \dfrac{q(a)}{\sep_x^\Delta(P) (a)}$.
\end{remark}

\begin{lemma} \label{lemmacoef}
Let $R$ be a $\Delta$-ring and let $A$ be a $\Delta$-algebra over $R$.
Let $A$ be $\Delta$-generated over $R$ by $\Delta$-algebraic over $R$ elements $a_1, \ldots, a_n$.
Then there exists an~invertible $\mathbb{Z}$-linear change of derivations transforming~$\Delta$ to~$\Delta^\ast$ such that~$a_1,\ldots,a_n$ are~$\Delta^\ast$-integral over~$R$.
\end{lemma}
 
\begin{proof}
Fix $1 \leqslant i \leqslant n$.
Since $a_i$ is $\Delta$-algebraic over $R$, there exists nonzero $f_i \in R[\Delta^\infty x]$ such that $f_i(a_i) = 0$.
We will choose this $f_i$ so that its highest (w.r.t.~\eqref{eq:ord}) derivative is minimal and, among such polynomials, the degree is minimal.
We will call such $f_i$ a \emph{minimal} polynomial for~$a_i$.

We introduce variables~$\lambda_2, \ldots, \lambda_m$ algebraically independent over~$A$ and extend the derivations from~$A$ to~$A[\lambda_2,\ldots,\lambda_m]$ by $\delta_i\lambda_j = 0$ for all~$i=1,\ldots,m$ and $j=2,\ldots,m$.
Consider a set of derivations $D := \{ \dd_1, \dd_2, \ldots, \dd_m\}$ defined by 
\[
 \dd_1 := \delta_1, \quad \dd_j := \delta_j + \lambda_i \delta_1 \text{ for } j = 2, \ldots, m.
\]

Consider any $1 \leqslant i \leqslant n$. 
We rewrite $f_i$ in terms of $D$ replacing $\delta_1$ with $\dd_1$ and $\delta_j$ with $\dd_j - \lambda_i \dd_1$ for $j = 2, \ldots, m$.
We denote the order of the highest derivative appearing in $f_i$ by $r_i$ and the partial derivative $\frac{\partial}{\partial (\dd_1^{r_i}x)}f_i$ by $s_i$.
We will show that $s_i(a_i) \neq 0$.
We write
\[
s_i(x) = \dfrac{\partial f_i}{\partial (\dd_1^{r_i} x)} = \sum\limits_{q_1 + \ldots + q_m = r_i} \lambda_2^{q_2}\ldots\lambda_m^{q_m}\dfrac{\partial f_i}{\partial(\delta_1^{q_1}\ldots \delta_m^{q_m}x)}.
\]

Due to the minimality of $f_i$ as a vanishing polynomial of $a$ and the algebraic independence of~$\lambda_j$, the latter expression does not vanish at $x = a_i$.
So, $s_i(a_i) \neq 0$.

Since, for every $1 \leqslant i \leqslant n$, $s_i(a_i)$ is a nonzero polynomial in $\lambda_2, \ldots, \lambda_m$ over $A$, it is possible to choose the values~$\lambda^\ast_2, \ldots, \lambda_m^\ast \in \mathbb{Z} \subset R$ so that neither of~$s_i(a_i)$ vanishes at $(\lambda_2^\ast, \ldots, \lambda_m^\ast)$.
Let $\Delta^\ast = \{\delta_1^\ast,\ldots,\delta_m^\ast\}$ be the result of plugging these values to $D$.
Then we have $\operatorname{sep}_x^{\Delta^\ast}f(a_i) = \dfrac{\partial f_i}{\partial ((\delta_1^{\ast})^{r_i} x)}(a_i) \neq 0$ for every $i = 1, \ldots, n$, so $a_1, \ldots, a_n$ are $\Delta^\ast$-integral over $R$.
\end{proof}

\begin{lemma} \label{lemmafinite}
Let $R$ be a $\Delta$-ring and let $A$ be a $\Delta$-algebra over $R$.
Assume that $A$ is a domain and is $\Delta$-generated by $a_1, \ldots, a_n$ over $R$ which are $\Delta$-integral over $R$.

Then there exists~$a\in A$ such that~$A[1/a]$ is finitely $\Delta_0$-generated over~$R$. 
\end{lemma}

\begin{proof}
We will prove the lemma by induction on the number~$n$ of~$\Delta$-generators of~$A$. If~$n=0$, then~$A=R$ and~$A$ is clearly finitely $\Delta_0$-generated.

Assume that the lemma is proved for all~extensions $\Delta$-generated by less than~$n$ elements.
Applying the induction hypothesis to $\Delta$-algebra~$A_0 := R[\Delta^\infty a_1, \ldots, \Delta^\infty a_{n - 1}]$, we obtain~$b_0 \in A_0$ such that~$A_0[1/b_0]$ is a finitely $\Delta_0$-generated~$R$-algebra.

Since $a_n$ is $\Delta$-integral over $R$, there exists~$P(x)\in R[\Delta^\infty x]$ such that $P(a_n) = 0$, the highest derivative in~$P$ is~$\delta_1^r x$, and~$b_2:=\sep^\Delta_x(P)(a_n) \neq 0$.
We claim that 
\begin{equation}\label{eq:fingen}
A\left[ \frac{1}{b_1b_2}\right] = A_0\left[ \frac{1}{b_1}, \Delta_0^\infty \left(\frac{1}{b_2}\right), \Delta_0^{\infty} (\delta_1^{\leqslant r}a_n)\right],
\end{equation}
where $\delta_1^{\leqslant r} a_n := \{a_n, \delta_1 a_n, \ldots, \delta_n^{r}a_n\}$.
Since $A_0[1/b_1]$ is finitely $\Delta_0$-generated over~$R$, this would imply that~$A[1/(b_1b_2)]$ is finitely $\Delta_0$-generated over~$R$ as~well.

In order to prove~\eqref{eq:fingen}, it is sufficient to show that the images of~$\{\delta_1^{\leqslant r} a_n, 1/b_2\}$ under $\delta_1$ belong to $B := A_0[1/b_1, \Delta_0^\infty (1/b_2), \Delta_0^\infty (\delta_1^{\leqslant r}a_n)]$. 
This is clear for~$\delta_1^{< r}a_n$, so it remains to show that~$\delta_1^{r+1}a_n,  \delta_1(1/b_2) \in B$:
\begin{itemize}
    \item For~$\delta_1^{r+1}a_n$ we use Remark~\ref{eqremark} to write
    \[
    \delta_1^{r+1}a_n = \dfrac{-q(a_n)}{\operatorname{sep}^{\Delta}_x (P) (a_n)} = \dfrac{-q(a_n)}{b_2} \in B, \text{ where } q \in R[\Delta_0^\infty (\delta_1^{\leqslant r}x)].
    \]
    \item For~$\delta_1(1/b_2)$, we observe that
\[
\delta_1 \left(\dfrac{1}{b_2}\right) \in \dfrac{1}{b_2^2}A_0[\Delta_0^\infty (\delta_1^{\leqslant r} a_n)] \subseteq B.\qedhere
\]
\end{itemize}
\end{proof}



\subsection{Proof of Seidenberg's Theorem}\label{sec:seidenberg}

\begin{lemma} \label{lemmamer}
Let~$W\subseteq \mathbb{C}^m$ be a domain and~$K$ be a countably $\Delta$-generated subfield of~$\Mer_m(W)$. 
Then there exist~$c\in \mathbb{C}$ and a~domain~$V \subseteq W \cap \{z_1 = c\}$ such that, for every $f \in K$, $f|_{\{z_1 = c\}}$ is a well-defined element of $\Mer_{m - 1}(V)$ and, therefore, the restriction to $\{z_1 = c\}$ defines a $\Delta_0$-embedding $K \to \Mer_{m - 1}(V)$.
\end{lemma}

\begin{proof}
  Let~$K$ be $\Delta$-generated by~$\{b_i\}_{i=1}^\infty$. 
  For every $i \geqslant 0$, denote by~$S_i$ the set of singularities of~$b_i$. 
  By definition, $S_i$ is a nowhere dense subset of~$\mathbb{C}^m$. 
  Therefore, the union~$S = \bigcup\limits_{i=1}^{\infty} S_i$ is~a~meagre set. 
  As~$W$ is a domain in~$\mathbb{C}^m$, the difference~$W \setminus S$ is non-empty.
  Choose any point $(w_1, \ldots, w_m) \in W \setminus S$. 
  Then all the restrictions of~$b_i$ to~$\{z_1 = w_1\}$ are holomorphic at~$(w_2,\ldots,w_m)$ and meromorphic in some vicinity~$V \subseteq W\cap \{t_1 = w_1\}$ of $(w_2,\ldots,w_m)$.
  Since every element $f \in K$ is a rational function in $b_i$'s and their partial derivatives, its restriction to $\{z_1 = w_1\}$ is also a well-defined meromorphic function on $V$.
\end{proof}

\begin{lemma}\label{lemmatransc}
Let $U \subseteq \mathbb{C}^m$ be a domain.
For every countably $\Delta$-generated $\Delta$-field $K \subseteq \Mer_m(U)$,
$\operatorname{difftrdeg}_K^\Delta \Mer_m(U)$ is infinite.
\end{lemma}

\begin{proof}
Suppose~$\operatorname{difftrdeg}_K^\Delta \mathcal{M}er_m(U) = l < \infty$. Let~$\tau_1,\ldots,\tau_l$ be a~$\Delta$-transcendence basis of~$\mathcal{M}er_m(U)$ over~$K$. Let~$L$ be a field~$\Delta$-generated by~$K$ and~$\tau_1,\ldots,\tau_l$. Note that~$L$ is still at most countably~$\Delta$-generated and~$\mathcal{M}er_m(U)$ is $\Delta$-algebraic over~$L$. Choose an arbitrary point~$c \in U$ and denote by~$F$ a subfield of~$\mathbb{C}$ generated by the values at~$c$ of those elements of~$L$ that are holomorphic at~$c$. Clearly~$F$ is at most countably generated and the transcendence degree of~$\mathbb{C}$ over~$F$ is infinite. Now any function in~$\mathcal{M}er_m(U)$ that is holomorphic at~$c$ and such that the set of values of its derivatives at~$c$ is~transcendental over~$F$ is~$\Delta$-transcendental over~$L$, which contradicts to the assumption that~$\mathcal{M}er_m(U)$ is $\Delta$-algebraic over~$L$.  
\end{proof}

\begin{notation}
  Let~$A$ be a~$\Delta$-algebra without zero divisors. 
  By~$\operatorname{Frac}(A)$ we denote the field of fractions of~$A$.
\end{notation}

We are now ready to prove Seidenberg's theorem.

\setcounter{theorem}{0}

\begin{theorem}[Seidenberg's embedding theorem]\label{main}
Let $W \subseteq \mathbb{C}^m$ be a domain and
let~$K \subseteq \mathcal{M}er_m(W)$ be at most countably $\Delta$-generated~$\Delta$-field (over $\mathbb{Q}$).
Let $L \supset K$ be a $\Delta$-field finitely $\Delta$-generated over~$K$. 

Then there exists a domain $U \subseteq W$ and a~$\Delta$-embedding~$f\colon L \rightarrow \mathcal{M}er_m(U)$ over $K$. 
\end{theorem}

\begin{proof}
We will first reduce the theorem to the case when $L$ is $\Delta$-algebraic over $K$.
Assume that it is not, and let $u_1, \ldots, u_\ell$ be a $\Delta$-transcendence basis of $L$ over $K$.
Lemma~\ref{lemmatransc} implies that there exist functions $f_1, \ldots, f_\ell \in \Mer_m(W)$ $\Delta$-transcendental over $K$.
Let $\Tilde{K}$ be a $\Delta$-field generated by $K$ and $f_1, \ldots, f_\ell$.
The embedding $K \to L$ can be extended to an embedding $\Tilde{K} \to L$ by sending $f_i$ to $u_i$ for every $1\leqslant i \leqslant \ell$.
Therefore, by replacing $K$ with $\Tilde{K}$ we will further assume that $L$ is $\Delta$-algebraic over $K$.

We~will prove the theorem by induction on~the number~$m$ of~derivations.
If~$m=0$, then~$L$ can be embedded into~$\mathbb{C}$ by \cite[Chapter V, Theorem 2.8]{Lang}.

Suppose~$m > 0$. 
Let~$a_1, \ldots, a_n$ be a set of $\Delta$-generators of $L$ over $K$.
Let $A := K[\Delta^\infty a_1, \ldots, \Delta^\infty a_n]$.
Since~$A$ is $\Delta$-algebraic over $K$, by Lemma~\ref{lemmacoef}, there exist and invertible $m\times m$ matrix $M$ over $\mathbb{Q}$ such that, for a new set of derivations 
\[
\Delta^\ast = \{\delta_1^\ast, \ldots, \delta_m^\ast\} := M\Delta,
\]
$a_1,\ldots,a_n$ are~$\Delta^\ast$-integral over~$K$.
Due to the invertibility of $M$, every $\Delta^\ast$-embedding $L \to \Mer_m(U)$ over $K$ yields a $\Delta$-embedding.
Therefore, by changing the coordinates in the space $\mathbb{C}^m$ from $(z_1, \ldots, z_m)$ to $M^{-1}(z_1, \ldots, z_m)$, we can further assume that $\Delta = \Delta^\ast$, so $a_1, \ldots, a_n$ are $\Delta$-integral over $K$.

Lemma \ref{lemmafinite} implies that there exists~$a \in A$ such that~$B:= A[a^{-1}]$ is finitely $\Delta_0$-generated. 
By $\Delta$-integrality of~$a_1,\ldots,a_n$ and Remark~\ref{eqremark}, for every $1 \leqslant i \leqslant n$, there exists a positive integer $r_i$ and a rational function $g_i \in K(\delta^\alpha y \mid \alpha <_{\grlex} (r, 0, \ldots, 0))$ such that $a_i$ satisfies
\begin{equation}\label{sys1}
    \delta^{r_i} a_i = g_i(a_i).
\end{equation}

Since $K$ is at most countably $\Delta$-generated,
Lemma \ref{lemmamer} implies that there exist~$w_1 \in \mathbb{C}$ and~$V\subseteq W \cap \{z_1 = w_1\} \subseteq \mathbb{C}^{m-1}$ such that the restriction to $\{z_1 = w_1\}$ induces a $\Delta_0$-embedding~$\rho\colon K \to \Mer_{m-1}(V)$.
We apply the induction hypothesis to $\Delta_0$-fields $\rho(K)$ and $\Frac(B) = L$.
This yields $\Delta_0$-embedding~$h: L \rightarrow \Mer_{m - 1} (\widetilde{V})$ for some~$\widetilde{V} \subseteq V$. 

Choose a point~$v = (w_2,\ldots,w_m)\in \widetilde{V}$ such that all the $h(a_i)$ are holomorphic at~$v$ and all the~$g_i(a_i)$ are holomorphic at~$w = (w_1,w_2,\ldots,w_m) \in W$. 
Consider the Taylor homomorphism $T_{h, w}\colon A\rightarrow \Mer_{m-1}(\widetilde{V})[[z_1]]$ defined as follows (see Definition~\ref{deftaylor}):
\[
  T_{h, w}(a) := \sum\limits_{k=0}^\infty h(\delta_1^k a)\dfrac{(z_1-w_1)^k}{k!} \quad\text{ for every }a \in A.
\]
Note that~$T_{h, w}$ is a~$\Delta$-homomorphism.

Fix $1 \leqslant i \leqslant n$. 
Since $a_1$ is a solution of~\eqref{sys1} and $T_{h, w}$ is a $\Delta$-homomorphism, $T_{h, w}$
is a formal power series solution of~$\delta_1^{r_i}y = g_i(y)$ corresponding to holomorphic initial conditions
\[
  y|_{z_1 = w_1} = h(a_i),\; (\delta_1 y)|_{z_1 = w_1} = h(\delta_1 a_i),\; \ldots, \; (\delta_1^{r_i - 1} y)|_{z_1 = w_1} = h(\delta_1^{r_i - 1} a_i).
\]
By the Cauchy-Kovalevskaya theorem, this solution is holomorphic in some vicinity~$U_i$ of~$w$.
We set $U := \bigcap_{i = 1}^n U_i$.
Thus,~$T_{h, w}$ induces a non-trivial~$\Delta$-homomorphism from~$A$ to~$\Mer_m(U)$. 
Since~$h$ is injective,~$T_{h, w}$ is also injective, so it can be extended to a $\Delta$-embedding $L \to \Mer_m(U)$ over $K$.

\end{proof}



\subsection{Proof of Ritt's theorem}\label{sec:ritt}

\begin{definition}[Differentially simple rings]
A~$\Delta$-ring~$R$ is called~\emph{$\Delta$-simple} if it contains no proper~$\Delta$-ideals.
\end{definition}

\begin{lemma} \label{lemmasimple}
Let~$A$ be a $\Delta$-simple $\Delta$-algebra $\Delta$-generated by~$a_1,\ldots,a_n$ over a $\Delta$-field K. 
Then $A$ does not contain zero divisors.

Furthermore, assume that there exists an integer $\ell$ such that
\begin{itemize}
    \item $a_1, \ldots, a_\ell$ form a $\Delta$-transcendence basis of $A$ over $K$;
    \item $a_{\ell + 1}, \ldots, a_n$ are $\Delta$-integral over $K[\Delta^\infty a_1, \ldots, \Delta^\infty a_\ell]$. 
\end{itemize}
Then~$A$ has finite $\Delta_0$-transcendence degree over~$K$.
In particular, $\ell = 0$.
\end{lemma}

\begin{proof}
Consider any non-zero (not necessarily differential) homomorphism~$\psi : A \rightarrow F$ ($F\supseteq K$ is a field) and the corresponding Taylor homomorphism~$T_{\psi, 0}\colon A \rightarrow F[[z_1,\ldots,z_m]]$, which is a~$\Delta$-homomorphism. 
Since~$A$ is~$\Delta$-simple, the kernel of~$T_{\psi, 0}$ is zero. Therefore,~$T_{\psi, 0}$ is a~$\Delta$-embedding of~$A$ into~$F[[z_1, \ldots, z_m]]$. 
Since~the latter does not contain zero divisors, the same is true for~$A$. 

Assume that~$A$ has infinite $\Delta_0$-transcendence degree, that is, $\ell > 0$. 
Since $a_{\ell + 1}, \ldots, a_n$ are $\Delta$-integral over $R := K[\Delta^\infty a_1, \ldots, \Delta^\infty a_\ell]$, Lemma~\ref{lemmafinite} implies that there exists an element~$b\in A$ such that~$A_0 := A[1/b]$ is a finitely $\Delta_0$-generated algebra over~$R$. 
Note that~$A_0$ is also $\Delta$-simple.
Let $A_0 = R[\Delta_0^\infty b_1, \ldots, \Delta_0^\infty b_s]$.
For every $j \geqslant 0$, consider $\Delta_0$-algebra
\[
  B_j := K[\Delta_0^\infty (\delta_1^{(<j)}a_1), \ldots, \Delta_0^\infty  (\delta_1^{(<j)}a_l), \Delta_0^\infty  b_1,\ldots, \Delta_0^\infty  b_s].
\]
For every~$j \geqslant 0$, we have
\[
    jl \leqslant \operatorname{difftrdeg}^{\Delta_0}_K B_j \leqslant jl+s.
\]
This inequality implies that there exists~$N$ such that, for every~$j>N$,~$\delta_1^j a_1,\ldots,\delta_1^j a_l$ are~$\Delta_0$-independent over~$B_j$. 
Consider any non-zero $\Delta_0$-homomorphism
\[
    \Tilde{\varphi}\colon B_N \rightarrow L,
\]
where~$L \supseteq K$ is a $\Delta_0$-field.
Due to the~$\Delta_0$-independence of the rest of the elements~$\delta^j_1 a_i$ for~$1\leqslant i\leqslant l$ and $j > N$, $\Tilde{\varphi}$ can be extended to a~homomorphism~$\varphi\colon A_0 \rightarrow L$ so that~$\varphi(\delta_1^j a_i) = 0$ for every~$1\leqslant i \leqslant l$ and~$j > N$.

Consider a Taylor homomorphism $T_{\varphi, 0} \colon A_0 \to L[[z]]$ with respect to $\delta_1$.
Since $\varphi$ was a $\Delta_0$-homomorphism, $T_{\varphi, 0}$ is a $\Delta$-homomorphism.
It remains to observe that the kernel of~$T_{\varphi, 0}$ contains~$\delta_1^{N+1}a_1,\ldots,\delta_1^{N+1}a_l$ contradicting to the fact that~$A_0$ is~$\Delta$-simple.
\end{proof}

\begin{theorem}[Ritt's theorem of zeroes]\label{Ritt}
Let $W \subseteq \mathbb{C}^m$ be a domain and
let~$K \subseteq \mathcal{M}er_m(W)$ be a~$\Delta$-field. 
Let~$A$ be a finitely generated~$\Delta$-algebra over~$K$. 

Then there exists a non-trivial $\Delta$-homomorphism~$f: A \rightarrow \mathcal{M}er_m(U)$ for some domain ${U \subseteq W \subseteq \mathbb{C}^m}$ such that $f(a)$ is~$\Delta$-algebraic over~$K$ for any~$a\in A$.
\end{theorem}

\begin{proof}
We can represent $A$ as $A = R / J$, where~$R := K[\Delta^\infty x_1,\ldots,\Delta^\infty x_n]$ and $J \subseteq R$ is a differential ideal. 
Since $R$ is a countable-dimensional $K$-space, $J$ can be generated by at most countable set of generators.
Pick any such set and denote the $\Delta$-field generated by the coefficients of the generators by $K_0$.
Then $K_0$ is countably $\Delta$-generated.
Let $R_0 := K_0[\Delta^\infty x_1,\ldots,\Delta^\infty x_n]$.
Since $J$ is defined over $K_0$, for $A_0 := R_0/(J \cap R_0)$, we have~$A = K \otimes_{K_0} A_0$.

Let~$I$ be a maximal differential ideal of~$A_0$, and consider the canonical projection~$\pi \colon A_0\rightarrow A_0/I$.
Let~$a_1,\ldots,a_n$ be a set of $\Delta$-generators of $A_0 / I$. 
Since~$A_0/I$ is differentially simple, by Lemma~\ref{lemmasimple}, $A_0/I$ does not have zero divisors and is $\Delta$-algebraic over $K_0$. 
We apply Theorem~\ref{main} to the $\Delta$-fields $K_0 \subseteq \Frac(A_0 / I)$ and obtain a $\Delta$-embedding $h \colon A_0/I \rightarrow \Mer_m(U)$ over $K$. 
Since~$A_0/I$ is~$\Delta$-algebraic over~$K_0$,~$h(a)$ is~also~$\Delta$-algebraic over~$K_0$ for any~$a \in A_0/I$. Let~$f_0 := h\circ\pi$, then $f_0(a)$ is~$\Delta$-algebraic over~$K_0$ for any~$a\in A_0$. Since~$K \subseteq \Mer_m(U)$, we can construct a nontrivial $\Delta$-homomorphism $f \colon K\otimes_{K_0} A_0 \to \Mer_m(U)$ as the tensor product of the embedding $K \to \Mer_m(U)$ and $f$.
The $\Delta$-algebraicity of the image of $f_0$ over $K_0$ implies the $\Delta$-algebraicity of the image of the image of $f$ over $K$.   
\end{proof}


\section{Remarks on the analytic spectrum}\label{sec:spec}

\begin{definition}[Analytic spectrum]
Consider $\mathbb{C}$ as a $\Delta$-field with the zero derivations.
Let $A$ be a finitely $\Delta$-generated $\Delta$-algebra over $\mathbb{C}$.
A homomorphism (not necessarily differential) $\psi \colon A \to \mathbb{C}$ of $\mathbb{C}$-algebras is called \emph{analytic} if, for every $a \in A$, the formal power series $T_{\psi, 0}(a)$ has a positive radius of convergence.
The set of the kernels of analytic $\mathbb{C}$-homomorphisms is called the \emph{analytic spectrum of}~$A$.  
\end{definition}

Corollary~\ref{cor:holomorphic} implies the following.

\begin{corollary}
Let~$A$ be a finitely generated ~$\Delta$-algebra with identity over~$\mathbb{C}$. Then the analytic spectrum of~$A$ is a Zarisky-dense subset of its maximal spectrum. 
\end{corollary}

\begin{proof}
Assume that analytic spectrum of~$A$ is not Zarisky-dense in $\operatorname{spec} A$. 
Then it is contained in some maximal proper closed subset~$F = \{M \in \operatorname{spec}A | a\in M\}$ for some non-nilpotent~$a\in A$.

Consider the localization~$A[a^{-1}]$ and a non-trivial~$\Delta$-homomorphism~$f\colon A[a^{-1}] \rightarrow \mathcal{O}_m(U)$
given by Corollary~\ref{cor:holomorphic}.
Then~$f(a)\neq 0$. 
We fix $u \in U$ such that $f(a)(u) \neq 0$ and consider a homomorphism~$g\colon A \rightarrow \mathbb{C}$ defined by $g(b) := f(b)(u)$ for every~$b\in A$. 
Note that~$I := \operatorname{Ker}(g)$ is an analytic ideal such that~$a\not\in I$, so we have arrived at the contradiction.
\end{proof}


\subsection*{Acknowledgements}
GP was partially supported by NSF grants DMS-1853482, DMS-1760448, and  DMS-1853650 and by the Paris Ile-de-France region.

\bibliographystyle{abbrvnat}
\bibliography{bibdata}

\end{document}